\documentclass[12pt]{amsart}
\usepackage{amsmath,amssymb,amsbsy,amsfonts,latexsym,amsopn,amstext,
                                               amsxtra,euscript,amscd}
\usepackage{url}
\usepackage[colorlinks,linkcolor=blue,anchorcolor=blue,citecolor=blue]{hyperref}
\usepackage{color}
\begin{document}

\newtheorem{theorem}{Theorem}
\newtheorem{lemma}[theorem]{Lemma}
\newtheorem{claim}[theorem]{Claim}
\newtheorem{cor}[theorem]{Corollary}
\newtheorem{prop}[theorem]{Proposition}
\newtheorem{definition}{Definition}
\newtheorem{question}[theorem]{Question}
\newcommand{\hh}{{{\mathrm h}}}

\def\sssum{\mathop{\sum\!\sum\!\sum}}
\def\ssum{\mathop{\sum\ldots \sum}}

\def \balpha{\boldsymbol\alpha}
\def \bbeta{\boldsymbol\beta}
\def \bgamma{{\boldsymbol\gamma}}
\def \bomega{\boldsymbol\omega}

\def\squareforqed{\hbox{\rlap{$\sqcap$}$\sqcup$}}
\def\qed{\ifmmode\squareforqed\else{\unskip\nobreak\hfil
\penalty50\hskip1em\null\nobreak\hfil\squareforqed
\parfillskip=0pt\finalhyphendemerits=0\endgraf}\fi}

\newfont{\teneufm}{eufm10}
\newfont{\seveneufm}{eufm7}
\newfont{\fiveeufm}{eufm5}
%
%
\newfam\eufmfam
     \textfont\eufmfam=\teneufm
\scriptfont\eufmfam=\seveneufm
     \scriptscriptfont\eufmfam=\fiveeufm
%
%
\def\frak#1{{\fam\eufmfam\relax#1}}

\def\fK{\mathfrak K}
\def\fT{\mathfrak{T}}

\def\fA{{\mathfrak A}}
\def\fB{{\mathfrak B}}
\def\fC{{\mathfrak C}}
\def\fD{{\mathfrak D}}

\def\eqref#1{(\ref{#1})}

\def\vec#1{\mathbf{#1}}
\def\dist{\mathrm{dist}}
\def\vol#1{\mathrm{vol}\,{#1}}

\def\squareforqed{\hbox{\rlap{$\sqcap$}$\sqcup$}}
\def\qed{\ifmmode\squareforqed\else{\unskip\nobreak\hfil
\penalty50\hskip1em\null\nobreak\hfil\squareforqed
\parfillskip=0pt\finalhyphendemerits=0\endgraf}\fi}

\def\sA{\mathscr A}
\def\sB{\mathscr B}
\def\sC{\mathscr C}
\def\sD{\Delta}
\def\sE{\mathscr E}
\def\sF{\mathscr F}
\def\sG{\mathscr G}
\def\sH{\mathscr H}
\def\sI{\mathscr I}
\def\sJ{\mathscr J}
\def\sK{\mathscr K}
\def\sL{\mathscr L}
\def\sM{\mathscr M}
\def\sN{\mathscr N}
\def\sO{\mathscr O}
\def\sP{\mathscr P}
\def\sQ{\mathscr Q}
\def\sR{\mathscr R}
\def\sS{\mathscr S}
\def\sU{\mathscr U}
\def\sT{\mathscr T}
\def\sV{\mathscr V}
\def\sW{\mathscr W}
\def\sX{\mathscr X}
\def\sY{\mathscr Y}
\def\sZ{\mathscr Z}

\def\cA{{\mathcal A}}
\def\cB{{\mathcal B}}
\def\cC{{\mathcal C}}
\def\cD{{\mathcal D}}
\def\cE{{\mathcal E}}
\def\cF{{\mathcal F}}
\def\cG{{\mathcal G}}
\def\cH{{\mathcal H}}
\def\cI{{\mathcal I}}
\def\cJ{{\mathcal J}}
\def\cK{{\mathcal K}}
\def\cL{{\mathcal L}}
\def\cM{{\mathcal M}}
\def\cN{{\mathcal N}}
\def\cO{{\mathcal O}}
\def\cP{{\mathcal P}}
\def\cQ{{\mathcal Q}}
\def\cR{{\mathcal R}}
\def\cS{{\mathcal S}}
\def\cT{{\mathcal T}}
\def\cU{{\mathcal U}}
\def\cV{{\mathcal V}}
\def\cW{{\mathcal W}}
\def\cX{{\mathcal X}}
\def\cY{{\mathcal Y}}
\def\cZ{{\mathcal Z}}
\newcommand{\rmod}[1]{\: \mbox{mod} \: #1}

\def\vr{\mathbf r}

\def\e{{\mathbf{\,e}}}
\def\ep{{\mathbf{\,e}}_p}
\def\em{{\mathbf{\,e}}_m}
\def\en{{\mathbf{\,e}}_n}

\def\Tr{{\mathrm{Tr}}}
\def\Nm{{\mathrm{Nm}}}

 \def\SS{{\mathbf{S}}}

\def\lcm{{\mathrm{lcm}}}

\def\({\left(}
\def\){\right)}
\def\fl#1{\left\lfloor#1\right\rfloor}
\def\rf#1{\left\lceil#1\right\rceil}

\def\mand{\qquad \mbox{and} \qquad}

         \newcommand{\comm}[1]{\marginpar{\vskip-\baselineskip
         \raggedright\footnotesize
\itshape\hrule\smallskip#1\par\smallskip\hrule}}




\hyphenation{re-pub-lished}

\parskip 4pt plus 2pt minus 2pt

\mathsurround=1pt

\def\bfdefault{b}
\overfullrule=5pt

\def \F{{\mathbb F}}
\def \K{{\mathbb K}}
\def \Z{{\mathbb Z}}
\def \Q{{\mathbb Q}}
\def \R{{\mathbb R}}
\def \C{{\\mathbb C}}
\def\Fp{\F_p}
\def \fp{\Fp^*}

\title[Dominating Sets in Circulant Graphs]{Constructing Dominating Sets in Circulant Graphs}

 \author[I. E. Shparlinski] {Igor E. Shparlinski}
 \thanks{This work was  supported in part by ARC Grant~DP140100118}

\address{Department of Pure Mathematics, University of New South Wales,
Sydney, NSW 2052, Australia}
\email{igor.shparlinski@unsw.edu.au}

\begin{abstract} We give an efficient construction of a reasonably small  dominating set in 
a circulant graph on $n$ notes and $k$ distinct chord lengths.    This result is based 
on bounds on some double exponential sums. 
\end{abstract}


\maketitle

\section{Introduction}

We recall that a subset $\cD \subseteq \cV$ of a graph $\cG  = (\cV,\cE)$ (directed or undirected) with the vertex set $\cV$ and the edge set $\cE$ 
is called {\it dominating\/} if for any $v\in \cV$ there is an edge $(u,v) \in \cE$ with $u \in \cD$. 

The smallest size of a dominating set of $\cG$ is called {\it the domination number\/} 
of $\cG$ and denoted by $\gamma(\cG)$. 

Here we investigate dominating sets of circulant graphs. We remark that although this 
direction has been studied by several authors, see~\cite{CSD,DeSe,KumMac,TaCheMut1,TaCheMut2,TaCheMut3}
and references therein, no general bound on the domination number of a circulant graph
is known.

For an integer $n\ge 2$ we use $\Z_n$  to denote the residue ring modulo
$n$ that we assume to be represented by the set $\{0,1, \ldots, n-1\}$. 
Let $\widetilde{\Z}_n$ be the set of non-zero elements of $\Z_n$.
Thus, for a prime $n=p$ we have  $\widetilde{\Z}_p = \Z_p^*$,
the set of invertible elements in $\Z_p$. 

A {\it circulant graph} is a directed $n$-vertex graph with an
automorphism that is an $n$-cycle. Circulant graphs may be constructed
as follows. Given a set $\cS  \subseteq \widetilde{\Z}_n$ 
we define the graph $\cC_n(\cS)$ to be the directed graph with the vertex set
$\Z_n$ where  for $i,j\in\Z_n$  there is an edge from $i$ to $j$  if and only
if $i-j\in \cS$. It is not difficult to see that $\cC_n(\cS)$ is an
$n$-vertex circulant graph of regularity $\# \cS$.

We say that  $\cS \subseteq  \widetilde{\Z}_n$ is {\it symmetric\/} if $s\in \cS$ 
if and only if $n-s \in \cS$. 
Then $\cC_n(\cS)$ is an undirected circulant graph. 
Clearly every symmetric set $\cS$ of cardinality $k$ can be represented as 
\begin{equation}
\label{eq:sym set}
\cS = \cT\, \bigcup\, (n -\cT) =  \{t_1, n-t_1, \ldots, t_m, n-t_m\}
\end{equation}
for some set $\cT = \{t_1, \ldots, t_m\} \subseteq \widetilde \Z_n$,  with 
$m = \rf{k/2}$ (for an odd $k$ we must have $n/2 \in \cS$ and thus $n$ 
has to be even).


Before we formulate our results, we recall that the notations $U = O(V)$,  $U \ll V$  and $V \gg U$ are all
equivalent to the assertion that the inequality $|U|\le c|V|$ holds for some
constant $c>0$.  Throughout the paper, the implied constants in the symbols `$O$',  `$\ll$' 
and `$\gg$'  are absolute.

As we obviously have $\(\#\cD +1\) \#\cS \ge n$  for 
any dominating set $\cD$, for  any  set $\cS \subseteq  \widetilde{\Z}_n$ of c
ardinality $k$, we obtain 
$$
 \gamma\(\cC_n(\cS)\) \ge \frac{n}{k} -1. 
$$
A random choice of $\cD$ leads to the bound 
\begin{equation}
\label{eq:RandChoice}
 \gamma\(\cC_n(\cS)\)\ll \frac{n  \log n }
{k } ,
\end{equation}
see~\cite[Corollary~3.2]{BJR} or~\cite[Theorem~2]{Lor}.
The proof of~\eqref{eq:RandChoice} given by Lorenz~\cite[Theorem~2]{Lor}
is probabilistic using a recursive choice of the dominating set $\cD$. It  can 
derandomised but it seems to lead to a construction with approximately $kn^{2+o(1)}$. 

Here we show how to find a reasonably small dominating set 
in time $ k^{1/2} n^{1+o(1)}$. 

\begin{theorem}
\label{thm:Dom} 
For  any  set $\cS \subseteq  \widetilde{\Z}_n$ of cardinality $k$, 
in time  $ k^{1/2} n^{1+o(1)}$, 
one can find a dominating set $\cD$ 
for  the graph $\cC_n(\cS)$  of size 
 $$
\# \cD\ll \frac{n   (\log n)^{5/2}}
{k^{1/2} \log \log n} . 
$$
\end{theorem}

One can naturally extend the definition of dominanting sets to {$r$-dominanting sets} of a a graph $\cG  = (\cV,\cE)$
 and say that 
$\cD_r \subseteq \cV$ is o {$r$-dominanting} an $r$-dominating set if i
for any $v\in \cV$ there is path of length at most $r$ originating some  $u \in \cD$. 

\begin{theorem}
\label{thm:Dom2} 
There is an absolute constant $C>0$ such that 
for  any   integer $k$ with 
$$
n \ge k \ge C n^{1/2}   \frac{(\log n)^{3}}
{\log \log n},
$$
in time  $ k^{1/2} n^{1+o(1)}$, 
one can find a set $\cD_2$ 
for  the graph $\cC_n(\cS)$  of size 
 $$
\# \cD_2\ll \frac{n^2   (\log n)^{5}}
{k^{2}( \log \log n)^2} , 
$$
which is a $2$-dominating set for  any graph $\cC_n(\cS)$  
with $\# \cS \ge k$.
\end{theorem}

We remark that an interesting feature of Theorem~\ref{thm:Dom2} is that the set $\cD_2$ 
is universal and does not depend on the set $\cS$. 

 \section{Preliminaries}
 
We fix  a positive integer parameter $L< n/2$ and let $\cL$ be the set of primes $\ell \in [L+1,2L]$
 with $\gcd(\ell, n ) = 1$.
 We define the set $\cW \subseteq \Z_n$ as 
\begin{equation}
\label{eq:set W}
 \cW = \{k/\ell \pmod n~:~ (k,\ell) \in [1, L]\times \cL\}. 
\end{equation}

 \begin{lemma}
\label{lem:Card}  
For $L< 0.5 n^{1/2}$ we have $\cW = L \#\cL$.
\end{lemma}

\begin{proof}  It is enough to show that 
$$
k_1/\ell_1 \not \equiv  k_2/\ell_2 \pmod n
$$
for any two distinct pairs $(k_1,\ell_1), (k_2,\ell_2)\in [1,L]\times \cL$. 
Assuming that this fails, we obtain 
$$
k_1\ell_2  \equiv  k_2\ell_1 \pmod n.
$$
Since $1 \le k_1\ell_2, k_2\ell_1 \le 2L^2 < n$ we conclude that 
$k_1\ell_2  =k_2\ell_1$. Since $\max\{k_1, k_2\}\le L < \min\{\ell_1, \ell_2\}$ 
and $\ell_1, \ell_2$ are primes, this implies $(k_1,\ell_1) =(k_2,\ell_2)$ and 
concludes the proof. 
\end{proof}

Let $\en(z) = \exp(2 \pi i z/n)$. 

We need the following bound of exponential 
sums, which is a modification of~\cite[Lemma~3]{Shp}. 

 \begin{lemma}
\label{lem:ExpSum}  
For any $a \in \widetilde{\Z}_n$ we have 
$$
\sum_{w \in \cW} \en(a w) \ll   L \frac{(\log n)^2}{\log \log n}.   
$$
\end{lemma}

\begin{proof}  By Lemma~\ref{lem:Card}, it is enough to show that 
$$
\sum_{k=1}^L \sum_{\ell \in \cL} \en(a k/\ell) \ll L \frac{(\log n)^2}{\log \log n}, 
$$
where the inversion in the argument of $\en$ is modulo $n$. 
Following the proof of~\cite[Lemma~3]{Shp}, we
define 
$$
I = \fl{\log(2n/L)} \mand J = \fl{\log (2n)}. 
$$

Furthermore, for a rational number $\alpha = u/v$ with $\gcd(v,n)=1$,
we denote by $\rho_n(\alpha)$ the unique integer $w$ with $w \equiv u/v \pmod n$
and $-n/2 < w \le n/2$ 
(we can assume that $n\ge 3$). 
Using the bound 
$$
\sum_{x=1}^{L} \en(\alpha x) 
\ll  \min \left\{L, \frac{n}{|\rho_n(\alpha)|}\right\},
$$
which holds for any rational $\alpha$ with the denominator which  is  relatively prime 
to $n$ (see~\cite[Bound~(8.6)]{IwKow}), we obtain a version 
of~\cite[Equation~(1)]{Shp}: 
\begin{equation}
\label{eq:W RT}
\sum_{k=1}^L \sum_{\ell \in \cL}   \en(a k/\ell)  \ll LR  
+ n \sum_{j = I+1}^J T_{j} e^{-j},
\end{equation}
where
\begin{equation*}
\begin{split}
&R =\# \left\{\ell \in \cL~:~ |\rho_n(a/\ell)| < e^{I} \right\},\\
&T_{j} =\# \left\{\ell \in \cL~:~ e^j  \le |\rho_n(a/\ell)|     < e^{j+1} \right\}, \quad j = I+1, \ldots, J.
\end{split}
\end{equation*}

We note that if $ |\rho_n(a/\ell)|  < Z$ then 
 $\ell z \equiv a \pmod n$ for some integer $z$ with 
$0 < |z| <   Z$. We can assume that $1 \le a \le n-1$. 
Thus $\ell z =a + n m$ for some integer $m$ with 
$|m| <2LZ/n$.
Hence there are at most $O\(LZ/n + 1\)$ 
possible values of $m$  and 
for each fixed $m \ll  LZ/n$ there are at most $O(\log n/\log \log n)$ 
primes $\ell$ dividing $a + n m \ne 0$. 
Therefore, we obtain the estimates
\begin{equation*}
\begin{split}
R &\ll \(e^I L /n + 1\)  \log n/\log \log n, \\
T_{j}&   \ll \(e^{j}  L /n + 1\)  \log n/\log \log n, \quad j = I+1, \ldots, J. 
\end{split}
\end{equation*}
In particular, recalling the definition of $I$, we see that 
\begin{equation*}
\begin{split}
R &\ll \log n/\log \log n, \\
T_{j}&   \ll\frac{e^{j}  L \log n}{n \log \log n}, \quad j = I+1, \ldots, J. 
\end{split}
\end{equation*}
Substituting these bounds in~\eqref{eq:W RT}, we obtain 
\begin{equation*}
\begin{split}
\sum_{k=1}^L   \sum_{\ell \in \cL}  \en(a k/\ell) 
&   \ll  L  \frac{\log n}{\log \log n}   + p\sum_{j = I+1}^J  \frac{e^{j}  L \log n}{n \log \log n}  e^{-j}\\
&  \ll   J L \frac{\log n}{\log \log n} \ll  L \frac{(\log n)^2}{\log \log n}, 
\end{split}
\end{equation*}
which concludes the proof. 
\end{proof}

It is interesting to note that neither the result nor the proof of Lemma~\ref{lem:ExpSum}  
depend on $\gcd(a,n)$ and only require $a \not \equiv 0 \pmod n$. 

 \begin{lemma}
\label{lem:Except Set}  
Let $\cW$ be given by~\eqref{eq:set W} and let $\cS\subseteq \Z_n$ be an arbitrary set. 
Then the set $\cU$ of $u \in \Z_n$ that cannot be represented as $u = s+w$ for 
$(s,w) \in \cS \times \cW$ is of cardinality
$$
\#\cU  \ll   \frac{n^2 (\log n)^4}
{ \#\cS (\# \cL)^2(\log \log n)^2}.
$$
\end{lemma}

\begin{proof}  
Using the orthogonality of exponential functions, 
the number $N$ of solutions to the equation 
$$
s+w -u= 0, \qquad (s,u,w) \in \cS  \times \cU  \times \cW, 
$$ 
(considered in the ring $\Z_n$) can be written as 
$$
N  = \sum_{s\in \cS}\sum_{u \in \cU} \sum_{w \in \cW} \frac{1}{n}\sum_{a\in \Z_n} \en(a(s+w-u)). 
$$
After changing the order of summation and separating the contribution $\#\cS\#\cU \#\cW/n$ corresponding to
$a = 0$, we obtain 
$$
N  - \frac{\#\cS\#\cU \#\cW}{n} =\frac{1}{n}\sum_{a\in\widetilde{\Z}_n}
 \sum_{s\in \cS}  \en(a  s) \sum_{u \in \cU}  \en(-a  u) \sum_{w \in \cW}  \en(a  w). 
$$
We now note that by the definition of $\cU$ we have $N = 0$.
Therefore 
$$
 \#\cS\#\cU \#\cW \le \sum_{a\in \widetilde{\Z}_n}\left|
 \sum_{s\in \cS} \en(a  s) \right| \left|\sum_{u \in \cU}  \en(-a  u) \right|
 \left| \sum_{w \in \cW}  \en(a  w)\right|.
$$
Using  Lemma~\ref{lem:Card},  we derive
\begin{equation}
\label{eq:AUW}
 \#\cS\#\cU \#\cW \ll  L \frac{(\log n)^2}{\log \log n}
 \sum_{a\in  \widetilde{\Z}_n}\left|
 \sum_{s\in \cS}  \en(a  s)\right| \left|\sum_{u \in \cU}  \en(a  u)  \right|. 
\end{equation}
Now, by the Cauchy inequality 
we obtain 
\begin{equation}
\label{eq:Cauchy}
\begin{split}
\(\sum_{a\in  \widetilde{\Z}_n}\left|  
 \sum_{s\in \cS}  \en(a  s)\right| \left|\sum_{u \in \cU}  \en(a  u)  \right|\)^2&\\
  \le  \sum_{a\in  \widetilde{\Z}_n} \left|  
 \sum_{s\in \cS}  \en(a  s)\right|^2   &\sum_{a\in  \widetilde{\Z}_n} 
 \left|\sum_{u \in \cU}  \en(a  u)  \right|^2. 
\end{split}
\end{equation}
Furthermore, expanding the summation over $a$ to $\Z_n$, we see that 
\begin{equation}
\label{eq:sum A}
\begin{split}
  \sum_{a\in  \widetilde{\Z}_n}& \left|   \sum_{s\in \cS}  \en(a  s)\right|^2   
 \le   \sum_{a\in \Z_n}  \left|   \sum_{s\in \cS}  \en(a  s)\right|^2   \\
&=  \sum_{a\in \Z_n}    \sum_{s,t\in \cS} \en(a  (s-t)) =
   \sum_{a,b\in \cS}  \sum_{a\in \Z_n}   \en(a  (s-t))  =  n\#\cS. 
\end{split}
\end{equation}
Similarly 
\begin{equation}
\label{eq:sum U}
\sum_{a\in  \widetilde{\Z}_n} 
 \left|\sum_{u \in \cU}  \en(a  u)  \right|^2 \le n\#\cU.
\end{equation}
Substituting~\eqref{eq:sum A} and~\eqref{eq:sum U} in~\eqref{eq:Cauchy}
and recalling~\eqref{eq:AUW} we derive 
$$
 \#\cS\#\cU \#\cW \ll  n L \sqrt{\#\cS\#\cU } \frac{(\log n)^2}{\log \log n}.
$$
Hence 
$$
 \#\cU \ll   \frac{n^2 L^2 (\log n)^4}
{ \#\cS(\#\cW)^2(\log \log n)^2}.
$$
It remains to recall that by  Lemma~\ref{lem:Card}  we have  
$\#\cW  = L \#\cL$.
\end{proof}

 \section{Proof of Theorem~\ref{thm:Dom}}
 
To prove the upper bound, we define $\lambda$ by the equation 
\begin{equation}
\label{eq:xi}
\frac{n^2  (\log \lambda)^2 (\log n)^4}
{k \lambda^2 (\log \log n)^2} = \frac{\lambda^2} {\log \lambda}. 
\end{equation}
We now set $L = \rf{\lambda}$ (and easily verify that $L < 0.5 n^{1/2}$ 
for a sufficiently large $n$) and  then define 
$$
\# \cD = \cU \cup \cW, 
$$
where $\cU$ is as in Lemma~\ref{lem:Except Set}  and $\cW$ is defined
by~\eqref{eq:set W}. Clearly $\cD$ is a dominating set of $\cC_n(\cS)$. 

We also note that since we always have $k < n$ the equation~\eqref{eq:xi} 
implies that $L \ge \lambda > n^{1/4}$. Since the number of distinct prime divisors
of $n$ is $O(\log n/\log \log n)$, by the prime number theorem, we obtain
\begin{equation}
\label{eq:Set L}
 \frac{\lambda} {\log \lambda} \gg \frac{L} {\log L} \gg   \#\cL \gg \frac{L} {\log L} \gg  \frac{\lambda} {\log \lambda},
\end{equation}
provided $n$ is large enough.

By  Lemmas~\ref{lem:Card}  and~\ref{lem:Except Set} 
$$
\# \cD  \ll  L \#\cL +  \frac{n^2 (\log n)^4}
{ \#\cS (\# L)^2(\log \log n)^2}.
$$
Using~\eqref{eq:Set L} and recalling the choice of $\lambda$ given by~\eqref{eq:xi}, we have
$$
\# \cD  \ll  \frac{\lambda^2} {\log \lambda} + \frac{n^2  (\log \lambda)^2 (\log n)^4}
{k \lambda^2 (\log \log n)^2} = 2 \frac{\lambda^2} {\log \lambda}.
$$
On the other hand, we derive from~\eqref{eq:xi} 
that 
$$
\frac{\lambda^4} {(\log \lambda)^2} = \frac{n^2   (\log n)^4 \log \lambda} {k(\log \log n)^2} . 
$$
Therefore, using $\lambda \le n$, we now derive
$$
\frac{\lambda^2} {\log \lambda} \le
  \frac{n   (\log n)^{5/2}}
{k^{1/2} \log \log n},
$$
which gives the desired upper bound on $\cD$. 

To see the time complexity bound, we first note that the set $\cW$ can be constructed in time 
$L^2n^{o(1)} = n^{1+o(1)}$, see~\cite{vzGG} for the background on the complexity of computation. 
Then, for each $w \in cW$, we  mark elements of $\Z_n$ of the form $w+S$ in time
$kn^{o(1)}$. After this we collect all unmarked elements in the set $\cU$, 
which concludes the proof.

 \section{Proof of Theorem~\ref{thm:Dom2}}

We now set 
\begin{equation}
\label{eq:L def}
L = \rf{ c  \frac{n(\log n)^3}{k\log \log n}}
\end{equation}
for a sufficiently large $c$ and then define 
$$
\# \cD_2 = \cW 
$$
where $\cW$ is given by~\eqref{eq:set W}. 

We now fix some set   $\cS  \subseteq \widetilde{\Z}_n$, for any $u \in \Z_n$ 
for the number $N(u)$ of solutions to the equation 
$$
s+t +w -u= 0, \qquad (s,t,w) \in \cS  \times \cS \times \cW, 
$$ 
(considered in the ring $\Z_n$) we have $N(u)>0$. 

Using the orthogonality of exponential functions, as in the proof of Lemma~\ref{lem:Except Set}
we write
$$
N(u)  = \sum_{s,t\in \cS}  \sum_{w \in \cW} \frac{1}{n}\sum_{a\in \Z_n} \en(a(s+t+w-u)). 
$$
Again,  changing the order of summation and separating the contribution $(\#\cS)^2 \#\cW/n$ corresponding to
$a = 0$, we obtain 
$$
N  - \frac{(\#\cS)^2 \#\cW}{n} =\frac{1}{n}\sum_{a\in\widetilde{\Z}_n}  \en(-a  u) 
 \sum_{s,t\in \cS}  \en(a  (s+t))  \sum_{w \in \cW}  \en(a  w). 
$$
Hence, using  Lemma~\ref{lem:Card},
\begin{equation*}
\begin{split}
\left| N(u)  - \frac{(\#\cS)^2 \#\cW}{n} \right|  & \le 
\frac{1}{n} \sum_{a\in  \widetilde{\Z}_n}\left|
 \sum_{s\in \cS}  \en(a  s)\right|^2 \left|\sum_{u \in \cU}  \en(a  u)  \right|\\
  & \ll  L \frac{(\log n)^2}{n \log \log n}
 \sum_{a\in  \widetilde{\Z}_n}\left|
 \sum_{s\in \cS}  \en(a  s)\right|^2. 
 \end{split}
\end{equation*}
Now, recalling~\eqref{eq:sum A}, we obtain 
\begin{equation}
\label{eq:Nu Asymp}
 N(u)  - \frac{(\#\cS)^2 \#\cW}{n}  \ll       L \#\cS \frac{(\log n)^2}{\log \log n}. 
\end{equation}
Thus, we see from~\eqref{eq:Nu Asymp} that there is a constant $c_0$ such that if 
\begin{equation}
\label{eq:Ineq}
 \frac{(\#\cS)^2 \#\cW}{n} > c_0    L \#\cS \frac{(\log n)^2}{\log \log n}.
\end{equation}
Assume that $L$ defined by~\eqref{eq:L def} satisfies the inequality
\begin{equation}
\label{eq:Ineq}
L \le 0.5 n^{1/2}, 
\end{equation}
thus  Lemma~\ref{lem:Card} applies.
Then, since $\# \cS \ge k$, and also 
using~\eqref{eq:Nu Asymp}, we see that it is enough 
to satisfy the condition 
$$
\#\cL > c_0   \frac{n (\log n)^2}{k\log \log n}.
$$
Since, by the prime number theorem $\#\cL \sim L/\log L$ (we again recall that 
$n$ has $O(\log n/\log \log n)$ distinct prime divisors), 
choosing a sufficient 
large $c$ in~\eqref{eq:L def}, we obtain the above inequality.  
We now also choose a sufficiently large $C$ in the condition of Theorem~\ref{thm:Dom2} 
so that~\eqref{eq:Ineq}
holds as well. 

It remains to note that all implied constants are effective and can easily be computed
explicitly. Hence $c$ and $C$ can also be explicitly computed leading to the desired 
algorithm.

\section{Comments} 

Clearly, the upper bound of Theorem~\ref{thm:Dom} is nontrivial 
if  
$$
k \ge C   \frac{(\log n)^{5}}{(\log \log n)^2}
$$ 
for some constant $C>0$. It is certainly interesting to lower this threshold.

We also note that the set $\cW$ which ``almost dominates'' $\cC_n(\cS)$ (that is, 
dominates all but $o(n)$ nodes) does not depend on $\cS$. 
In fact, Lemma~\ref{lem:Except Set} implies that for any function $\psi(z) \to \infty$ 
as $z \to \infty$, one can construct 
such a universal set $\cW$ of size 
$$
\# \cW \le \psi(n)   \frac{n (\log n)^3}
{ k^{1/2} \log \log n},
$$
which almost dominates all graphs $\cC_n(\cS)$ with $\# \cS\ge k$. 

\section*{Acknowledgements}

The author would like to thank Imre Ruzsa and Ilya Shkredov  for their valuable 
comments,  in particular,  for pointing out the bound~\eqref{eq:RandChoice}
and the references~\cite{BJR,Lor}. 

During the preparation of this paper,
the  was supported in part by ARC grant~DP140100118.

\end{document}